\documentclass[11pt]{amsart}
\usepackage{hyperref}
\usepackage[psamsfonts]{amssymb}
\usepackage{amsmath,amsfonts,amssymb,amscd}
\usepackage{graphicx}
\usepackage{latexsym,epsfig}
\usepackage[latin1]{inputenc}
\usepackage[T1]{fontenc}
\usepackage{enumerate}
\usepackage{amsbsy}
\usepackage{graphics}

\newtheorem{theorem}{Theorem}[section]
\newtheorem{corollary}[theorem]{Corollary}
\newtheorem{lemma}[theorem]{Lemma}
\newtheorem{prop}[theorem]{Proposition}
\newtheorem{problem}[theorem]{Problem}

\newtheorem{remark}[theorem]{Remark}
\newtheorem{conjecture}[theorem]{Conjecture}

\newtheorem*{ack}{Acknowledgements}
\newcommand{\Aut}{\operatorname{Aut}}
\newcommand{\Inn}{\operatorname{Inn}}
\newcommand{\Za}{\operatorname{Z} }
\newcommand{\id}{\operatorname{id} }
\begin{document}

\title[Structure of twin groups]{Structural aspects of twin and pure twin groups}

\author{Valeriy BARDAKOV}
\address{Sobolev Institute of Mathematics, Novosibirsk, 630090, Russia. Novosibirsk State University and Agrarian university, Novosibirsk, Russia.}
\email{bardakov@math.nsc.ru}

\author{Mahender SINGH}
\address{Department of Mathematical Sciences, Indian Institute of Science Education and Research (IISER) Mohali, Sector 81, S. A. S. Nagar, P. O. Manauli, Punjab 140306, India.}
\email{mahender@iisermohali.ac.in}

\author{Andrei VESNIN}
\address{Tomsk State University, Tomsk, 634050, Russia, and Sobolev Institute of Mathematics, Novosibirsk, 630090, Russia.}
\email{vesnin@math.nsc.ru}

\subjclass[2010]{Primary 57M27; Secondary 57M25}

\keywords{Coxeter group, doodle, Eilenberg--Maclane space, free group, hyperbolic plane, pure twin group, twin group}

\begin{abstract}
The twin group $T_n$ is a Coxeter group generated by $n-1$ involutions and the pure twin group $PT_n$ is the kernel of the natural surjection of $T_n$ onto the symmetric group on $n$ letters. In this paper, we investigate structural aspects of twin and pure twin groups.  We prove that the twin group $T_n$ decomposes into a  free product with amalgamation for $n>4$. It is shown that the pure twin group $PT_n$ is free for $n=3,4$, and not free for $n\ge 6$. We determine a generating set for $PT_n$, and give an upper bound for its rank. We also construct a natural faithful representation of $T_4$ into $\Aut(F_7)$. In the end, we propose virtual and welded analogues of these groups and some directions for future work.
\end{abstract}

\maketitle

\section{Introduction}
Twin groups $T_n$, $n \ge 2$, are a class of Coxeter groups generated by $n-1$ involutions. These groups first showed appearance in the works of Shabat and Voevodsky \cite{Shabat-Voevodsky, Voevodsky} under the name Grothendieck cartographical groups. Later, these groups appeared in the work of Khovanov \cite{Kh96} on real $K(\pi, 1)$ arrangements, who called them twin groups,  and investigated them further in~\cite{Kh97}.

\par
Twin groups have a geometrical interpretation  \cite{Kh96, Kh97} similar to the one for classical braid groups. Consider two parallel lines, say $y = 0$ and $y = 1$, on the plane $\mathbb{R}^2$ with $n$ marked points on each line. Consider the set of configurations of $n$ arcs in the strip $\mathbb{R} \times [0,1]$ connecting $n$ marked points on line $y=1$ to those on the line $y=0$ such that each arc is monotonic and no three arcs have a point in common. Two such configurations are called equivalent if one can be deformed into the other by a homotopy of arcs keeping the end points of the arcs fixed throughout the homotopy, and an equivalence class is called a \textit{twin}. The product of two twins on $n$ arcs can be defined by placing one on top of the other and rescaling the interval to unit length.  This operation turns  the set of all twins on $n$ arcs into a group which is isomorphic to the group $T_n$. The generators $s_i$ are geometrically represented by configurations of the type as shown in Figure~\ref{fig1}.
 \begin{figure}[!ht]
 \begin{center}
\includegraphics[height=2.7cm, width=7cm]{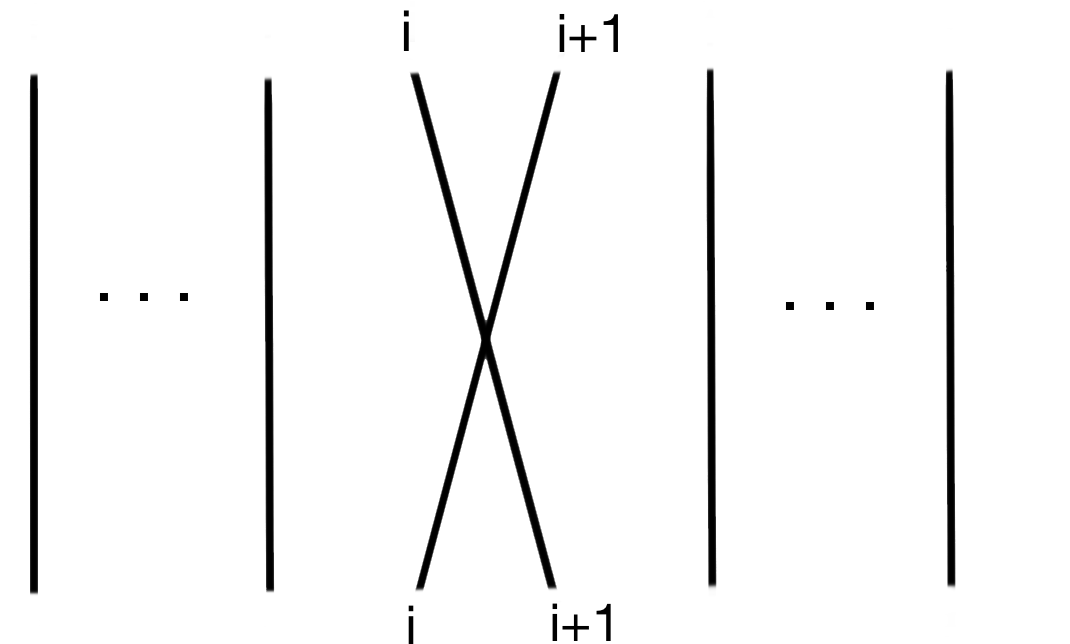}
\end{center}
\caption{The generator $s_i$.} \label{fig1}
\end{figure}

The closure of a twin is defined analogous to closure of a braid.  A doodle on a closed surface is a collection of finitely many piecewise-linear closed curves without triple intersections. The concept first appeared in the work of Fenn and Taylor \cite{FT}. The idea is similar to that of a classical link.  It is evident that closure of a twin gives a doodle. In \cite{Kh97}, Khovanov proved that every oriented doodle on a 2-sphere is the closure of a twin. Also, to each doodle, he associated a group called the fundamental group of the doodle, which plays the role of the fundamental group of a link complement. These constructions have been generalised by Bartholomew-Fenn-Kamada-Kamada \cite{BFKK}, wherein they consider collection of immersed circles in closed oriented surfaces of arbitrary genus. A Markov theorem for doodles on the 2-sphere has been established recently by Gotin \cite{Gotin}.
\par

The twin group $T_n$ is generated by $n-1$ involutions $s_{1}, s_{2}, \ldots, s_{n-1}$. The pure twin  group $PT_n$ is defined as the kernel of the natural homomorphism from the twin group $T_n$ to the symmetric group $S_n$ on the set $\{1, \ldots, n \}$, where the homomorphism maps the twin $s_i$ to the transposition $(i, i + 1)$. A nice topological interpretation of the pure twin group $PT_n$ is also known \cite{Kh96}. Consider the space
$$
X_n = \mathbb{R}^n \setminus \big\{ (x_1, x_2, \ldots, x_n) \in \mathbb{R}^n  \, | \, x_i = x_j = x_k, \quad i \not= j \not= k \not= i \big\},
$$
the complement of triple diagonals $x_{i} = x_{j} = x_{k}$. In \cite{Kh96}, Khovanov proved that the fundamental group $\pi_1(X_n)$ is isomorphic to pure twin group $PT_n$. Before that, Bj\"{o}rner and Welker \cite{BW} had investigated cohomology of these spaces and shown that each $H^i(X_n, \mathbb{Z})$ is free. They had also computed the ranks in some cases.
\par

The purpose of this paper is to investigate largely unexplored structural aspects of twin groups and pure twin groups. The paper is organised as follows. In Section~\ref{sec2}, we prove that the twin group $T_n$ decomposes into a semi-direct product for $n> 2$ (Proposition \ref{twin-semi-direct-product}), and a  free product with amalgamation for $n>4$ (Theorem~\ref{t2.3}). In Section~\ref{sec3}, we discuss pure twin groups. We prove that $PT_n$ is free for $n=3,4$ (Theorem \ref{pure-twin-n23}), and not free for $n\ge 6$ (Theorem~\ref{pure-twin-not-free}). Though our approach is mostly algebraic, we also present a geometric proof of freeness of $PT_4$ by showing that it can be realised as the fundamental group of the 8-punctured 2-sphere. In Section \ref{sec4}, we determine a generating set for $PT_n$ (Theorem \ref{pure-twin-gen}), and give an upper bound for the rank of $PT_n$ (Theorem \ref{rank-pure-twin}). We also show that there is a natural faithful representation of $T_4$ into $\Aut(F_7)$ (Proposition \ref{pt4-faithful}). Finally, in Section \ref{sec5}, we introduce virtual twin groups $VT_n$, welded twin groups $WT_n$, and formulate some problems for future work.
\par

Throughout the paper, we use the notation $a^b:=b^{-1}ab$ and $[a,b] := a^{-1} b^{-1} a b$.
\vspace*{1mm}

\section{Twin groups} \label{sec2}
Let $n \ge 2 $ be an integer. The \emph{twin group} $T_n$ is generated by the elements $s_1, s_2, \ldots, s_{n-1}$, and defined by the relations
\begin{eqnarray}\label{r1}
 s_{i}^2 & =& 1  \qquad \mbox{for}~ i = 1, 2, \ldots,n-1,
\end{eqnarray}
and
\begin{eqnarray}\label{r2}
  s_{i} s_{j} & =& s_{j} s_{i} \qquad \mbox{for}~ |i-j|\geq 2.
\end{eqnarray}
In particular, $$T_2 = \big\langle s_1~|~s_1^2 = 1 \big\rangle = \mathbb{Z}_2$$ is the cyclic group of order 2, and
$$
T_3 = \big\langle s_1, s_2~|~s_1^2 = s_2^2 = 1 \big\rangle = \mathbb{Z}_2 * \mathbb{Z}_2
$$
is the infinite dihedral group. Braid relations are missing in twin groups since triple intersections are not allowed for doodles.

\subsection{Reidemeister--Schreier method} We recall the well-known Reidemeister--Schreier me\-thod~\cite[Theorem 2.6]{MKS} that yields a presentation for a finite index subgroup $H$ of a finitely presented group $G=\langle X~|~ \mathcal{R} \rangle$. Consider the Schreier set  $\Lambda$ of coset representatives of $H$ in $G$. For each element $g \in G$, let $\overline{g}$ denote the unique coset representative of the coset of $g$ in the Schreier set $\Lambda$. Then by Reidemeister--Schreier Theorem \cite{MKS}, the subgroup $H$ is generated by the set
$$\big\{S_{\lambda, a}= (\lambda a)(\overline{\lambda a})^{-1}~|~\lambda \in \Lambda~\textrm{and}~a \in X \big\}.$$
\par

Further, the defining relations for $H$ are
$$\big\{ \tau(\lambda r \lambda^{-1})~|~ \lambda \in \Lambda~\mathrm{and}~r\in \mathcal{R} \big\},$$
where $\tau$ is Reidemeister's transformation, which  maps every nonempty word $x_{i_1}^{\varepsilon_1} \ldots x_{i_p}^{\varepsilon_p}$ $(\varepsilon_j=\pm 1)$ to a word in  $S_{K_\alpha, a_\nu}$ by the rule:
$$
\tau (x_{i_1}^{\varepsilon_1}\ldots x_{i_p}^{\varepsilon_p})=  S_{K_{i_1}, x_{i_1}}^{\varepsilon_1} \ldots S_{K_{i_p}, x_{i_p}}^{\varepsilon_p},
$$
where $K_{i_j}=\overline{x_{i_1}^{\varepsilon_1}\ldots x_{i_{j-1}}^{\varepsilon_{j-1}}}$, if $\varepsilon_j=1$, and $K_{i_j}=\overline{x_{i_1}^{\varepsilon_1}\ldots x_{i_{j}}^{\varepsilon_{j}}}$, if $\varepsilon_j=-1$.
\vspace*{1mm}

\subsection{Decomposition of $T_n$ into semi-direct product} It is easy to check that there is a surjective homomorphism $$\varphi_n : T_n \longrightarrow T_{n-1},$$ which maps $s_{n-1}$ to 1 and  $s_{i}$ to $s_i$ for all $i\neq n-1$. From the geometrical point of view $\varphi_n$ is the deletion of the $n$-th strand in the braid-type geometric presentation of twines.

Denote by $D_n = \text{\rm Ker}(\varphi_n)$, the kernel of the homomorphism $\varphi_n$.

\begin{prop}The following properties hold:\label{twin-semi-direct-product}
\begin{enumerate}
\item $T_n = D_n \leftthreetimes T_{n-1}$ for $n >2$.
\item $D_3 = \mathbb{Z}_2 * \mathbb{Z}_2 \cong T_3$.
\item $D_4 = \big\langle a_k, k \in \mathbb{Z} ~|~a^2_k = 1, k \in \mathbb{Z} \big\rangle$, the free product of  $\mathbb{Z}_2$ indexed by integers.
\end{enumerate}
\end{prop}

\begin{proof}
Recall that $T_{n-1}$ is a subgroup of $T_n$. In fact, the map $T_{n-1} \to  T_n$ given by $s_i \to s_i$, $1 \le i \le n-2$, is a splitting of the extension $$1 \to D_n \to T_n \to  T_{n-1} \to 1.$$ Thus, $T_n = D_n \leftthreetimes T_{n-1}$, which is assertion (1).
\par

By definition we have $T_3 = \langle s_1, s_2 ~|~s_1^2 = s_2^2 = 1 \rangle$ and $|T_{2}| = 2$. Then  $D_3 = \text{\rm Ker}\{ \varphi_3 : T_3 \longrightarrow T_2 \}$ has index $2$ in $T_3$. Consider  $\Omega_3 = \{ 1, s_1 \}$ as a set of coset representatives of $D_3$ in $T_3$. By Reidemeister--Schreier  method the group $D_3$ is generated by elements
$$
S_{\lambda, a} = \lambda a\, (\overline{\lambda a})^{-1}, \qquad \lambda \in \Omega_3, \quad a \in \{ s_1, s_2 \}.
$$
A direct computation yields
$$
S_{1,s_1} = 1, \qquad S_{1,s_2} = s_2, \qquad S_{s_1,s_1} = 1, \qquad S_{s_1,s_2} = s_1 s_2 s_1.
$$
Thus, $D_3$ is generated by elements $b = s_2$ and $c = s_1 s_2 s_1$. Further, it is easy to see that $D_3$ is defined by the relations $b^2 = c^2 = 1$, and hence $D_3 \cong \mathbb{Z}_2 * \mathbb{Z}_2$, which is assertion (2).
\par

Every element of $T_3$ can be written in the form $s_1^{\varepsilon} (s_1 s_2)^k$ for some $\varepsilon \in \{ 0, 1 \}$ and $k \in \mathbb{Z}$. Let $\Omega_4$ denotes the set of words of this form and consider it as a set of coset representatives of $D_4$ in $T_4$. Then $D_4$ is generated by elements
$$
S_{\lambda, a} = \lambda a \,(\overline{\lambda a})^{-1}, \qquad \lambda = s_1^{\varepsilon} (s_1 s_2)^k, \quad \varepsilon \in \{ 0, 1 \}, \quad k \in \mathbb{Z}, \quad a \in \{ s_1, s_2, s_3 \}.
$$
It is easy to check that $S_{\lambda, a} = 1$ for all $\lambda \in \Omega_4$ and $a \in \{ s_1, s_2 \}$. If $a = s_3$, then
$$
S_{\lambda, s_3} = s_1^{\varepsilon} (s_1 s_2)^k s_3  (s_1 s_2)^{-k} s_1^{\varepsilon}.
$$
Thus, in fact, $D_4$ is generated by elements
$$
a_{\varepsilon, k} = s_1^{\varepsilon} (s_1 s_2)^k s_3  (s_1 s_2)^{-k} s_1^{\varepsilon},
$$
where $\varepsilon \in \{ 0, 1 \}$ and $k \in \mathbb{Z}$. By Reidemeister--Schreier  method $D_4$ has defining  relations
$$
\tau (\lambda r \lambda^{-1}) = 1,
$$
where $r$ is a relation of $T_4$ and $\tau$ is a rewriting process.
\par

For the relation $r = [s_1, s_3]$  of $T_4$, we get
$$
\lambda [s_1, s_3] \lambda^{-1} = \lambda s_1 s_3 s_1 \lambda^{-1} ~ \lambda s_3 \lambda.
$$
For the case $\varepsilon = 1$ we obtain
$$
\tau \big(\lambda [s_1, s_3] \lambda^{-1}\big) = a_{0,-k} a_{1,k} = 1,
$$
and for the case $\varepsilon = 0$ we obtain
$$
\tau \big(\lambda [s_1, s_3] \lambda^{-1}\big) = a_{1,-k} a_{0,k} = 1.
$$
These two types of relations can be used to eliminate all generators of the type $a_{1,k}$, and hence $D_4$ is generated by elements $a_k = a_{0,k}$.
\par

For the relation $r = s_3^2$ of $T_{4}$ we have
$$
\tau (\lambda s_3^2 \lambda^{-1}) = a_{\varepsilon, k}^2,
$$
and therefore $a_k^2 = 1$.
\par

For the relations $r = s_1^2$ and $r = s_2^2$, the relations $\tau (\lambda r \lambda^{-1})=1$ are trivial since the generators $S_{\lambda, s_1}$ and $S_{\lambda, s_2}$ are trivial.  Thus, $D_4$ is generated by $a_k$ with defining relations $a_k^2 = 1$ for $k \in \mathbb{Z}$. This completes the proof of assertion (3).
\end{proof}

\subsection{Decomposition of $T_n$ into free product with amalgamation} In this subsection, we prove that $T_3$ and $T_4$ decompose into a free product, and $T_n$ decompose into a free product with amalgamation for $n>4$.

It is not difficult to find the structure of $T_n$ and its commutator subgroup $T'_n$ for small values of $n$.

\begin{prop} \label{p2.2} For $n=3, 4$, $T_{n}$ and $T'_{n}$ have the following description:
\begin{enumerate}
\item $T_3  \cong \mathbb{Z}_2 * \mathbb{Z}_2$.
\item $ T_3' = \big\langle [s_1, s_2] = (s_1 s_2)^2 \big\rangle \cong \mathbb{Z}$.
\item $T_4  \cong  \left( \mathbb{Z}_2 \times \mathbb{Z}_2 \right) * \mathbb{Z}_2$.
\item $T_4' = \big\langle [s_1, s_2], [s_3, s_2], [s_1 s_3, s_2] \big\rangle \cong F_3$.
\end{enumerate}
\end{prop}

\begin{proof}
The decomposition of $T_3$ and $T_4$ follows from the definition. More precisely, we can write
\begin{eqnarray*}
T_4 & =  &  \big\langle s_1, s_2, s_3 ~|~s_1^2 = s_2^2 = s_3^2 = 1, s_1 s_3  = s_3 s_1  \big\rangle \\
& = & \left(\langle s_1, s_3 ~|~s_1^2  = s_3^2 = 1, s_1 s_3  = s_3 s_1 \rangle \right) * \langle s_2~|~s_2^2 = 1 \rangle  \\
& \cong & \left( \mathbb{Z}_2 \times \mathbb{Z}_2 \right) * \mathbb{Z}_2.
\end{eqnarray*}

The results on the commutator subgroups follow from the following fact (see \cite[p.197]{MKS}): the commutator subgroup of the free product $A * B$ of abelian groups $A$ and $B$ is the subgroup of mixed commutators $[A, B]$ which is the kernel of the homomorphism $A * B \longrightarrow A \times B$. Further, $[A, B]$ is free with basis consisting of elements $[a, b]$, where $1 \not= a \in A$ and $1 \not= b \in B$.
\end{proof}

In general, the following decomposition holds.

\begin{theorem} \label{t2.3}
If $n > 4$, then the twin group $T_n$ has the following description as a free product with amalgamation:
$$
T_n \cong \left( \mathbb{Z}_2 \times T_{n-2} \right) \, *_{T_{n-3}} \left( \mathbb{Z}_2 \times T_{n-3} \right).
$$
\end{theorem}

\begin{proof}
Let $A = \langle s_1, s_3, s_4, \ldots, s_{n-1} \rangle$ be a subgroup of $T_n$ generated by the same set as $T_n$ except the generator  $s_2$. Then $A$ is a standard subgroup of $T_n$, i.~e. it is generated by a subset of the generating set of $T_n$. Hence $A$ has the presentation
\begin{eqnarray*}
A  &= & \big\langle s_1, s_3, s_4, \ldots, s_{n-1} ~|~ s_i^2 =  [s_1, s_j] = [s_j, s_k] = 1;\\
& & i = 1, 3, 4, \ldots, n-1; \quad  j = 3, 4, \ldots, n-1; \quad  j < k \leq n-1; \quad k-j > 1  \big\rangle\\
& \cong &  \mathbb{Z}_2 \times T_{n-2},
\end{eqnarray*}
where $\langle s_1 \rangle\cong \mathbb{Z}_2$ and the subgroup of $A$ that  is generated by $s_3, s_4, \ldots, s_{n-1}$ is isomorphic to $T_{n-2}$.

Consider the group
\begin{eqnarray*}
B & =&   \big\langle s_2, t_4, t_5, \ldots, t_{n-1} ~|~ s_2^2 = t_i^2 = [s_2, t_i] = [t_i, t_j] = 1;\\
& & i = 4, 5, \ldots, n-1; \quad i < j \leq n-1;  \quad j-i > 1  \big\rangle \\
& \cong & \mathbb{Z}_2 \times T_{n-3}.
\end{eqnarray*}

Then there exists an embedding $B \longrightarrow T_n$ defined on generators by the rules
$$
s_2 \mapsto s_2, \quad t_i \mapsto s_i,  \quad i = 4, 5, \ldots, n-1.
$$
The preceding embedding induces an isomorphism between the subgroup $A_{1} = \langle s_4, s_5, \ldots, s_{n-1} \rangle$ of $A$ and the subgroup $B_{1} = \langle t_4, t_5, \ldots, t_{n-1} \rangle$ of $B$ defined by $s_{i} \mapsto t_{i}$ for $i=4, \ldots, n-1$. It is easy to see that $A_{1} \cong B_{1} \cong T_{n-3}$, and hence
$$
T_n \cong A \, *_{T_{n-3}} B,
$$
which is the desired decomposition.
\end{proof}
\par

As a consequence, we obtain the following result.
\par

\begin{corollary}\label{center-twin}
The center of the twin group $T_n$ is trivial for $n>2$.
\end{corollary}

\begin{proof}
By Proposition \ref{p2.2}, $T_3$ and $T_4$ are non-trivial free products, and hence have trivial centers. By \cite[Lemma 2.5]{Hou}, if $G = G_1 *_H G_2$ is a free product of two groups with amalgamation such that $G_1 \not= H$ and $G_2 \not= H$, then $\Za(G) \leq \Za(H)$. This together with Theorem \ref{t2.3} implies that $\Za(T_n) \le\Za(T_{n-3})$ for $n > 4$. Thus, to conclude that $\Za(T_n)=1$ for all $n>2$, it only suffices to show that $\Za(T_5)=1$.
\par

It follows from Theorem \ref{t2.3} that
$$
T_5 = A *_{T_2} B,
$$
where $A = \langle s_1, s_3, s_4 \rangle \cong \langle s_1 \rangle \times \langle s_3, s_4 \rangle \cong \mathbb{Z}_2 \times T_3$ and $B = \langle s_2, t_4 \rangle  \cong \mathbb{Z}_2 \times \mathbb{Z}_2$. The amalgamated subgroup  $T_2 \cong \langle s_4 \rangle \cong \langle t_4 \rangle \cong \mathbb{Z}_2$, and hence $\Za(T_5)$ is either trivial, or isomorphic to $\mathbb{Z}_2$ with generator $s_4 = t_4$. Since the subgroup $\langle s_3, s_4 \rangle$ is isomorphic to $T_3$, it follows that $s_4$ does not commute with $s_3$, and hence $\Za(T_5)$ is trivial.
\end{proof}
\vspace*{1mm}


\section{Pure twin groups} \label{sec3}

Let $n \ge 2$ be an integer. Then the \emph{pure twin group} $PT_{n}$ is defined as the kernel of the homomorphism
from the twin group to the symmetric group
$$
\pi : T_n \longrightarrow S_n,
$$
which maps the generator $s_i$ to $(i,i + 1)$ for $i=1, 2, \ldots, n-1$.  For example, elements of the type $(s_i s_{i+1})^3$ lie in $PT_n$, and geometrically represented in Figure~\ref{fig2}.
 \begin{figure}[!ht]
 \begin{center}
\includegraphics[height=4cm]{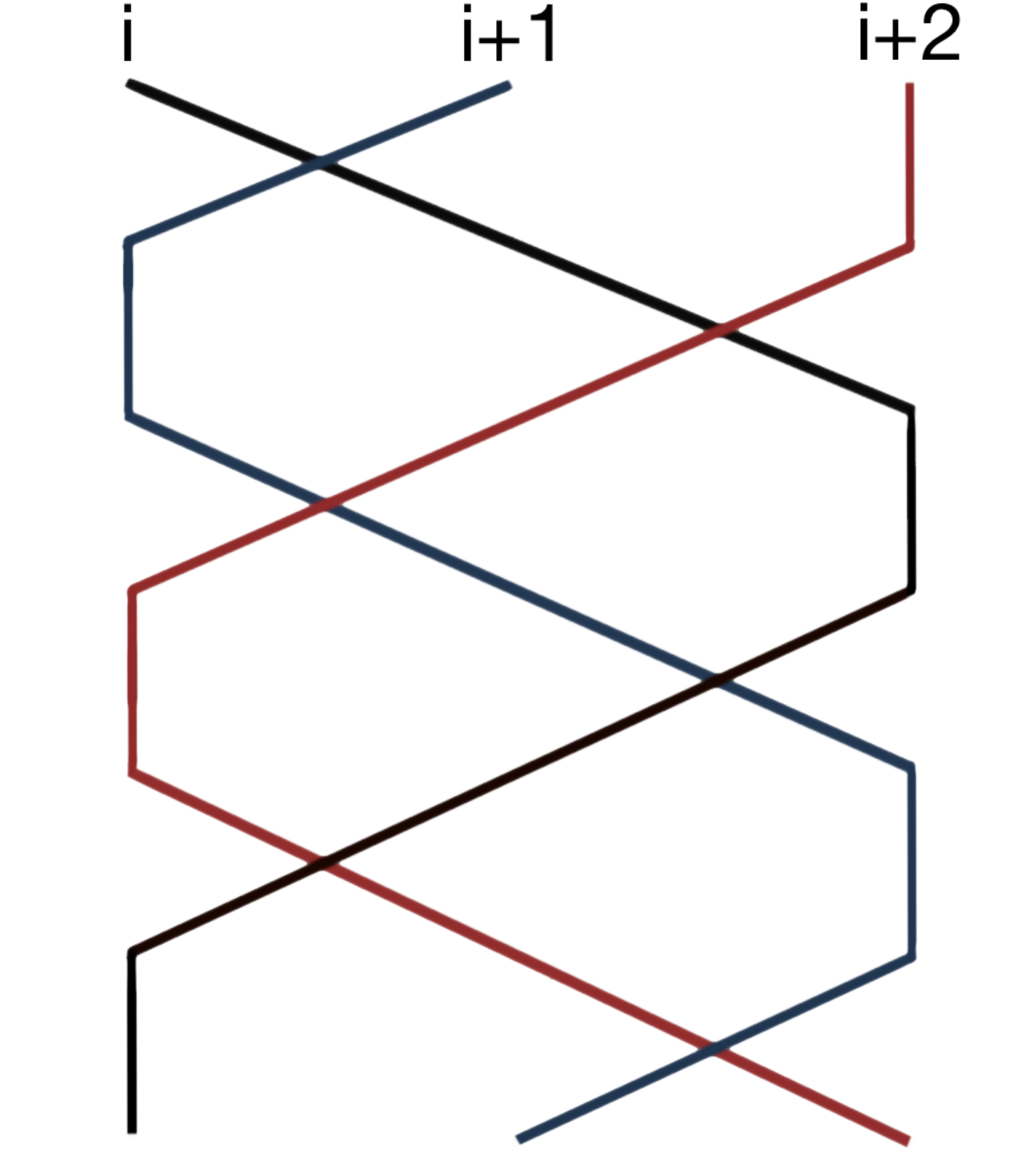}
\end{center}
\caption{Pure twin $(s_is_{i+1})^3$.} \label{fig2}
\end{figure}
The closure of the twin in Figure~\ref{fig2} is the Borromean doodle as shown in Figure~\ref{fig3}.
 \begin{figure}[!ht]
 \begin{center}
\includegraphics[height=4cm, width=5cm]{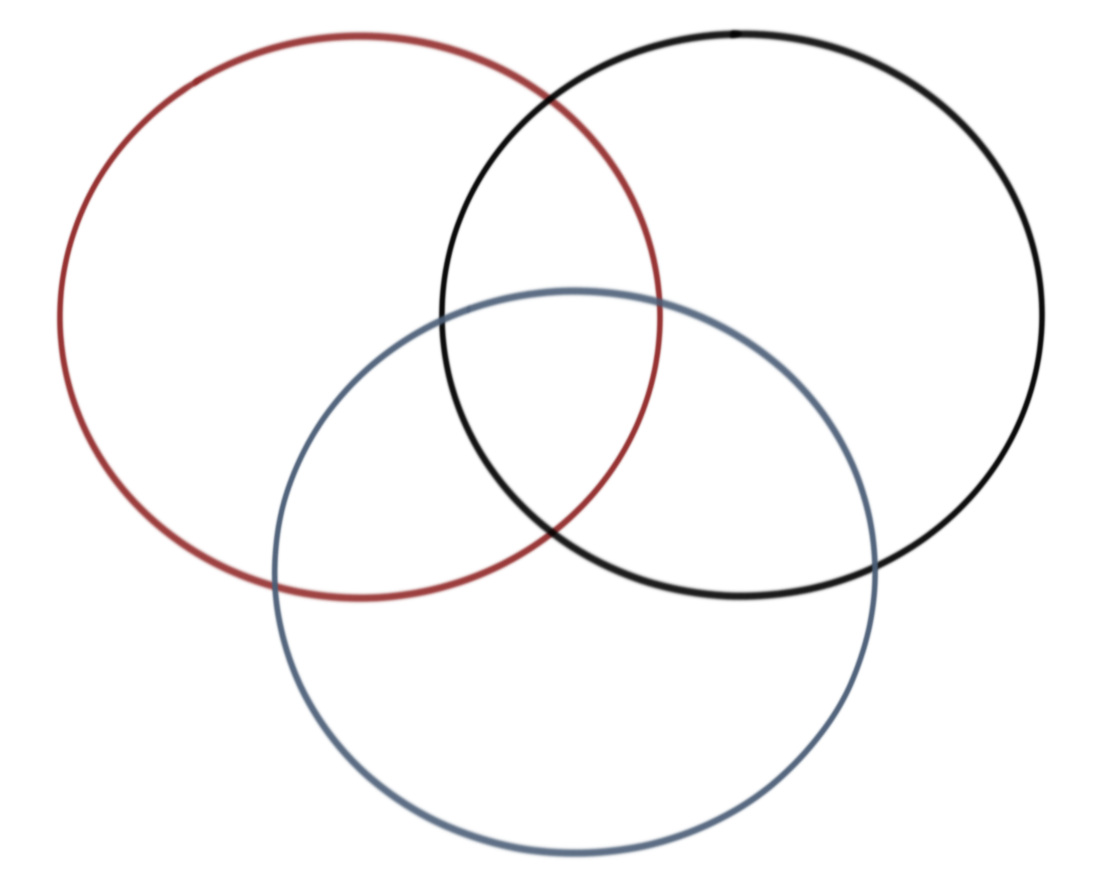}
\end{center}
\caption{Borromean doodle.} \label{fig3}
\end{figure}

Observe that for $n=2$ we have $T_2 \cong S_2$, and hence $PT_2 = 1$.  In this section, we find a generating set for $PT_n$ for $n>2$.  This is achieved by using Reidemeister--Schreier method as recalled in Section \ref{sec2}. Consider the Schreier set of coset representatives of $PT_n$ in $T_n$ given by $$\Lambda_n= \big\{ m_{1,i_1} m_{2,i_2} \ldots m_{n-1,i_{n-1}}~|~m_{k,i_k} = s_k s_{k-1} \ldots s_{i_k+1}~\mathrm{for~each}~1 \le k \le n-1~\mathrm{and}~0 \le i_k < k \big\}.$$
We set $m_{k,k} = 1$ for $1 \le k \le n-1$. For an element $x \in T_n$, let $\overline{x}$ denote the unique coset representative of the coset of $x$ in the Schreier set $\Lambda_n$. Then the pure twin group $PT_n$ is generated by the set
$$\big\{S_{\lambda, a}= (\lambda a)(\overline{\lambda a})^{-1}~|~\lambda \in \Lambda_n~\textrm{and}~a \in \{s_1, s_1, \ldots, s_{n-1} \} \big\}$$
and has relations
$$\big\{ \tau(\lambda r \lambda^{-1})~|~ \lambda \in \Lambda_n~\mathrm{and}~r~ \mathrm{a~ defining~ relation~ in}~T_n \big\},$$
where $\tau$ is Reidemeister's transformation.
\par

\subsection{Structure of $PT_n$}

For $n=3,4$, we determine the complete structure of $PT_n$.

\begin{theorem} \label{pure-twin-n23}
$PT_3 \cong \mathbb{Z}$ and $PT_4 \cong F_7$.
\end{theorem}

\begin{proof}
By Reidemeister--Schreier  method \cite{MKS} the group $PT_3$ is generated by elements
$$
S_{\lambda,a} = \lambda a ~ (\overline{\lambda a})^{-1},\qquad \lambda \in \Lambda_3, \qquad a \in \{ s_1, s_{2} \},
$$
where
$$
\Lambda_3 = \{ 1, s_1, s_2, s_1 s_2, s_2 s_1, s_1 s_2 s_1 \}
$$
is the set of right coset representatives of the group $PT_3$ in $T_3$. It is easy to see that among elements of the form $S_{\lambda,a}$ the only non-trivial elements are
$$
S_{s_2s_1,s_2} = (s_2 s_1)^3~\textrm{and}~S_{s_1s_2s_1,s_2} = (s_1 s_2)^3,
$$
which are inverses of each other. Hence, $PT_3$ is generated by one element.
\par

Next we find the defining relations of $PT_3$.  Considering the relation $r_2 = s_2^2$ we get
$$
s_2 s_1 r_2 s_1^{-1} s_2^{-1} = S_{s_2s_1,s_2} S_{s_1s_2s_1,s_2},
$$
and
$$
s_1 s_2 s_1 r_2 s_1^{-1} s_2^{-1} s_1^{-1} = S_{s_1s_2s_1,s_2} S_{s_2s_1,s_2}.
$$
Hence, $S_{s_2s_1,s_2} = S_{s_1s_2s_1,s_2}^{-1}$ and
$$
PT_3 = \langle (s_1 s_2)^3\rangle \cong \mathbb{Z}.
$$
\par

Now we consider  $n=4$. The group $PT_4$ is generated by elements
$$
S_{\lambda,a} = \lambda a ~ (\overline{\lambda a})^{-1},\quad \lambda \in \Lambda_4,\quad  a \in \{ s_1, s_{2}, s_3 \},
$$
where
$$
\Lambda_4 = \Lambda_3 \, \cup \, \Lambda_3 s_3 \, \cup \, \Lambda_3 s_3 s_2 \, \cup \, \Lambda_3 s_3 s_2 s_1
$$
is the set of right coset representatives of $PT_4$ in $T_4$. Straightforward calculations show that the following generators of $PT_4$ that do not lie in $PT_3$ are non-trivial:
$$
\begin{gathered}
S_{s_3s_2,s_3} = (s_3 s_2)^3,  \quad S_{s_1s_3s_2,s_3} = \left( (s_3 s_2)^3 \right)^{s_1}, \quad S_{s_2s_3s_2,s_3} = (s_2 s_3)^3, \\
S_{s_1s_2s_3s_2,s_3} = \left( (s_2 s_3)^3 \right)^{s_1}, \quad  S_{s_2s_1s_3s_2,s_3} = (s_2 s_1)^3\big((s_2 s_3)^3\big)^{s_1 s_2 s_1} , \\
S_{s_1s_2s_1s_3s_2,s_3} =  (s_1 s_2)^3\big((s_2 s_3)^3\big)^{s_1s_2}, \quad  S_{s_3s_2s_1,s_2}  = \left( (s_2 s_1)^3 \right)^{s_3}, \\
S_{s_3s_2s_1,s_3} =  (s_3 s_2)^3, \quad  S_{s_1s_3s_2s_1,s_2} = \left( (s_1 s_2)^3 \right)^{s_3}, \quad  S_{s_1s_3s_2s_1,s_3} = \left( (s_3 s_2)^3 \right)^{s_1}, \\ S_{s_2s_3s_2s_1,s_2} = \left( (s_2 s_1)^3 \right)^{s_3 s_2}, \quad S_{s_2s_3s_2s_1,s_3} =  (s_2 s_3)^3, \\
S_{s_1s_2s_3s_2s_1,s_2} = \left( (s_2 s_1)^3 \right)^{s_3 s_2 s_1}, \quad S_{s_1s_2s_3s_2s_1,s_3} = \left( (s_2 s_3)^3 \right)^{s_1}, \\
S_{s_2s_1s_3s_2s_1,s_2} = \big( (s_1 s_2)^3\big)^{s_3s_2}, \quad  S_{s_2s_1s_3s_2s_1,s_3} = (s_2 s_1)^3\big((s_2 s_3)^3\big)^{s_1s_2s_1}, \\ S_{s_1s_2s_1s_3s_2s_1,s_2} = \left( (s_1 s_2)^3 \right)^{s_3 s_2 s_1}, \quad
S_{s_1s_2s_1s_3s_2s_1,s_3} =  (s_1 s_2)^3\big((s_2 s_3)^3\big)^{s_1s_2}.
\end{gathered}
$$
After removing the redundant generators from the preceding list and using the fact that the element $(s_1 s_2)^3 \in PT_3$ is also a generator of $PT_4$, the complete list of generators of $PT_4$ can be rewritten as
$$
(s_1 s_2)^3,\quad \left( (s_1 s_2)^3 \right)^{s_3}, \quad \big( (s_1 s_2)^3\big)^{s_3s_2},\quad \left( (s_1 s_2)^3 \right)^{s_3 s_2 s_1},$$
$$ \quad (s_2 s_3)^3,  \quad  \left((s_2 s_3)^3 \right)^{s_1}, \quad \big((s_2 s_3)^3\big)^{s_1s_2}, \quad \big((s_2 s_3)^3\big)^{s_1 s_2 s_1}.$$

But, we have
\begin{equation}\label{equation-generalize}
(s_2 s_1)^3 \big((s_2 s_3)^3\big)^{s_1 s_2 s_1} (s_1 s_2)^3 = \big((s_3 s_2)^3\big)^{s_1s_2},
\end{equation}
and hence $PT_4$ is generated by 7 elements. Considering the defining relations, we see that $PT_4 \cong F_7$, where $F_7$ is the free group on the elements
$$
(s_1 s_2)^3,\quad \left( (s_1 s_2)^3 \right)^{s_3}, \quad \big( (s_1 s_2)^3\big)^{s_3s_2},\quad \left( (s_1 s_2)^3 \right)^{s_3 s_2 s_1},$$
$$ \quad (s_2 s_3)^3,  \quad  \left((s_2 s_3)^3 \right)^{s_1}, \quad \big((s_2 s_3)^3\big)^{s_1s_2}.$$
This completes the proof.
\end{proof}
\par

\begin{theorem}\label{pure-twin-not-free}
The pure twin group $PT_n$ is not free for $n \ge 6$. Further, $PT_n$ is torsion-free for $n \ge 3$.
\end{theorem}

\begin{proof}
In \cite{Kh96}, Khovanov proved that $PT_n$ is isomorphic to the fundamental group of the space $X_n$ which is an Eilenberg--Maclane space. Further, Bj\"{o}rner and Welker \cite{BW} proved that each $H^i
(X_n, \mathbb{Z})$ is free and $H^i (X_n, \mathbb{Z}) \neq 0$ iff $0 \le i \le n/3$. If $PT_n$ is free, then $X_n$ is homotopy equivalent to a wedge of circles, and hence $H^i (X_n, \mathbb{Z}) \neq 0$ iff $0 \le i < 2$. Thus, if $PT_n$ is free, then $n<6$.
\par

The second assertion follows from Theorem \ref{pure-twin-n23} for $n=3,4$. For the general case, recall from \cite{BW} that the Eilenberg--Maclane space $X_n$ of $PT_n$ is finite dimensional. Therefore, $PT_n$ is of finite cohomological dimension, and hence torsion-free.
\end{proof}
\par

\begin{remark}
Bj\"{o}rner and Welker \cite{BW} proved that $H^1(X_n, \mathbb{Z})$ has rank $\sum_{i=3}^n \binom{n}{i}\binom{i-1}{2}$. For $n=3, 4$, the formula yields the ranks 1, 7, respectively, which agrees with Theorem \ref{pure-twin-n23}. For $n=5$, $H^1 (X_5, \mathbb{Z})$ is a free abelian group of rank 31 and $H^i (X_5, \mathbb{Z}) = 0$ for $i \ge 2$.
\end{remark}
\par
In view of the preceding remark we propose the following:

\par

\begin{conjecture}
The pure twin group $PT_5$ is a free group of rank 31.
\end{conjecture}
\vspace*{1mm}

\subsection{A geometric proof of freeness of $PT_4$}

We demonstrate some geometrical properties of $PT_{4}$ and obtain another proof of its freeness. More precisely, we show that $PT_4$ can be realized as the fundamental group of the $8$-punctured 2-sphere.  As in the proof of Theorem~\ref{pure-twin-n23}, the group $PT_{4}$ is generated by elements
$$
S_{\lambda, a} = (\lambda  a)  (\overline{\lambda a})^{-1}, \quad \lambda \in \Lambda_{4}, \quad a \in \{s_{1}, s_{2}, s_{3} \},
$$
where
\begin{equation}\label{geom-pt4}
\Lambda_{4} = \Lambda_{3} \cup \Lambda_{3}  s_{3} \cup \Lambda_{3}  (s_{3} s_{2}) \cup \Lambda_{3}  (s_{3} s_{2} s_{1})
\end{equation}
and
$$
\Lambda_{3} = \big\{ s_{0} = 1, s_{1}, s_{2}, s_{1} s_{2}, s_{2} s_{1},  s_{1} s_{2} s_{1} \big\}.
$$

Observe that $T_{4}$ is a Coxeter group generated by reflections in sides of a triangle $D_{0}$ with angles $\pi/2$, $0$ and $0$ which can be realised in a hyperbolic plane $\mathbb H^{2}$, see Figure~\ref{fig4} for the picture in the upper half-plane model. Here two vertices lie in $\partial \mathbb H^{2}$. Such a vertex is usually called an \emph{ideal} vertex.
\begin{figure}[!ht]
\centering
\unitlength=0.3mm
\begin{picture}(0,110)(0,-10)
\thicklines
\qbezier(-80,0)(80,0)(80,0)
\put(0,-10){\makebox(0,0)[cc]{$0$}}
\put(40,-10){\makebox(0,0)[cc]{$1$}}
\put(80,-10){\makebox(0,0)[cc]{$x$}}
\put(75,0){\vector(1,0){10}}
\qbezier(0,40)(0,40)(0,100)
\qbezier(40,0)(40,0)(40,100)
\qbezier(0,40)(32,32)(40,0)
\qbezier(0,44)(0,44)(4,44)
\qbezier(4,44)(4,44)(4,40)
\put(20,60){\makebox(0,0)[cc]{$D_{0}$}}
\put(-10,70){\makebox(0,0)[cc]{$s_{1}$}}
\put(50,70){\makebox(0,0)[cc]{$s_{2}$}}
\put(20,20){\makebox(0,0)[cc]{$s_{3}$}}
\end{picture}
\caption{Triangle $D_{0}$.} \label{fig4}
\end{figure}
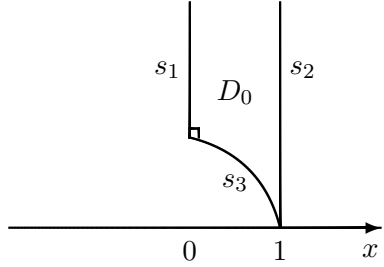

In this case $s_{1}$ and $s_{3}$ are reflections in sides of $D_{0}$ which form an angle $\pi/2$, and $s_{2}$ is the reflection in the third side that meets the two other sides at the point at the absolute $\partial \mathbb H^{2}$. Denote by $D_{0}^{1}, D_{0}^{2}$ and $D_{0}^{3}$ sides of $D_0$ corresponding to reflections $s_{1}$, $s_{2}$ and $s_{3}$, respectively.

Since $D_{0}$ is a fundamental domain of $T_{4}$ in $\mathbb H^{2}$, a fundamental domain of its subgroup $PT_{4}$ consists of $|S_{4}|=24$ copies of $D_{0}$ glued all together by the action of $\Lambda_{4}$. We denote a fundamental domain of $PT_{4}$ by $\mathcal {PD}$:
$$
\mathcal {PD} = \bigcup_{g \in \Lambda_{4}} ~g(D_{0}).
$$
For a generator $s_{i}$ and a domain $D_{k}$, we denote $s_{i}(D_{k}) = D_{ik}$; for two generators $s_{i}$ and $s_{j}$ and a domain $D_{k}$, we denote $s_{i} s_{j} (D_{k}) = s_{j} (s_{i} (D_{k})) = s_{j} (D_{ik}) = D_{jik}$. Further these notations are generalised in the natural way. We denote
\begin{eqnarray*}
\mathcal D_{0} = \bigcup_{g \in \Lambda_{3}} g(D_{0})  & = &  \big\{ D_{0}, D_{1} = s_{1} (D_{0}), D_{2} = s_{2} (D_{0}), D_{21} = s_{1} s_{2} (D_{0}) = s_{2} (s_{1} (D_{0})), \cr  & &  D_{12} = s_{2} s_{1} (D_{0}) = s_{1} (s_{2} (D_{0})), D_{121} = s_{1} s_{2} s_{1} (D_{0}) = s_{1} (s_{2} (s_{1} (D_{0}))) \big\}.
\end{eqnarray*}

Since $\Lambda_{4}$ is a union of four sets as in \eqref{geom-pt4}, we present $\mathcal {PD}$ as the union of the following four 6-element sets
$$
\mathcal {PD} = \mathcal D_{0} \cup s_{3 }\mathcal D_{0}  \cup s_{3} s_{2} \mathcal D_{0} \cup s_{3} s_{2} s_{1} \mathcal D_{0} = \mathcal D_{0} \cup \mathcal D_{3} \cup \mathcal D_{23} \cup \mathcal D_{123},
$$
where each set consists of 6 triangles $D_{j}$ with multi-indices $j$ given as follows:
$$
\mathcal D_{3} = \big\{ D_{3}, D_{31} , D_{32} , D_{312} , D_{321} , D_{3121} \big\},
$$
$$
\mathcal D_{23} = \big\{ D_{23}, D_{231} , D_{232} , D_{2312} , D_{2321} , D_{23121} \big\},
$$
$$
\mathcal D_{123} = \big\{ D_{123}, D_{1231} , D_{1232} , D_{12312} , D_{12321} , D_{123121} \big\}.
$$
Analogous to the case of $D_{0}$, for each triangle $D_{j} \in \mathcal {PD}$ with multi-index  $j$, we denote by $D_{j}^{1}$, $D_{j}^{2}$ and $D_{j}^{3}$ its sides, where we follow the same rule as for $D$ but considering instead of $s_{1}$, $s_{2}$ and $s_{3}$ their conjugates.

Any element $S_{\lambda,a} = (\lambda  a)  (\overline{\lambda a})^{-1} \in \Lambda_{4}$ can be written in the form $s_{j} s_{i} s_{k}^{-1}$, where $s_{i} \in \{s_{1}, s_{2},  s_{3}\}$, and $j$, $k$ are such multi-indices, that $s_{j}, s_{k} \in \Lambda_{4}$. These elements identify side $D_{k}^{i}$ with side $D_{j}^{i}$.  In particular, for $a \in \{ s_{1}, s_{2} \}$ and $\lambda \in \Lambda_{3}$, we have the following side identifications:
\begin{equation} \label{glue1}
D_{1}^{1} \to D_{0}^{1}, \quad D_{12}^{1} \to D_{2}^{1}, \quad D_{121}^{1} \to D_{21}^{1}, \quad \text{and} \quad
D_{2}^{2} \to D_{0}^{2}, \quad D_{21}^{2} \to D_{1}^{2}, \quad D_{121}^{2} \to D_{12}^{2}.
\end{equation}
The set
$$
\mathcal D_{0} = \big\{ S_{\lambda, a} (D_{0}) \, | \, \lambda \in \Lambda_{3}, a \in \{ s_{1}, s_{2}\} \big\}
$$
is presented in Figure~\ref{fig5}, where left and right sides should be identified by $(s_{2} s_{1}) s_{2} (s_{1} s_{2} s_{1})^{-1}$.
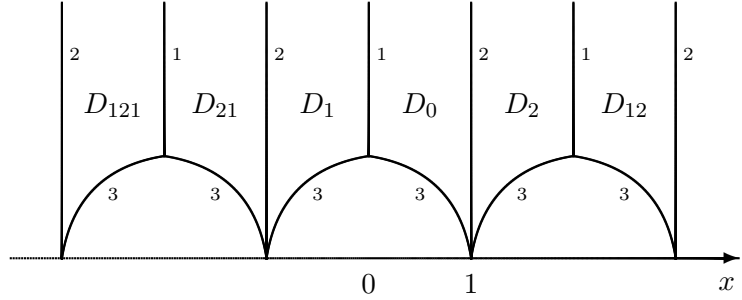
\begin{figure}[!ht]
\centering
\unitlength=0.34mm
\begin{picture}(0,120)(0,-10)
\thicklines
\qbezier(-140,0)(140,0)(140,0)
\put(0,-10){\makebox(0,0)[cc]{$0$}}
\put(40,-10){\makebox(0,0)[cc]{$1$}}
\put(140,-10){\makebox(0,0)[cc]{$x$}}
\put(135,0){\vector(1,0){10}}
\qbezier(-120,0)(-120,0)(-120,100)
\qbezier(-80,40)(-80,40)(-80,100)
\qbezier(-40,0)(-40,0)(-40,100)
\qbezier(0,40)(0,40)(0,100)
\qbezier(40,0)(40,0)(40,100)
\qbezier(80,40)(80,40)(80,100)
\qbezier(120,0)(120,0)(120,100)
\qbezier(-120,0)(-114.6,34.6)(-80,40)
\qbezier(-40,0)(-34.6,34.6)(0,40)
\qbezier(40,0)(45.4,34.6)(80,40)
\qbezier(-80,40)(-45.4,34.6)(-40,0)
\qbezier(0,40)(34.6,34.6)(40,0)
\qbezier(80,40)(114.6,34.6)(120,0)
\put(-100,60){\makebox(0,0)[cc]{$D_{121}$}}
\put(-60,60){\makebox(0,0)[cc]{$D_{21}$}}
\put(-20,60){\makebox(0,0)[cc]{$D_{1}$}}
\put(20,60){\makebox(0,0)[cc]{$D_{0}$}}
\put(60,60){\makebox(0,0)[cc]{$D_{2}$}}
\put(100,60){\makebox(0,0)[cc]{$D_{12}$}}
\put(-115,80){\makebox(0,0)[cc]{\tiny 2}}
\put(-75,80){\makebox(0,0)[cc]{\tiny 1}}
\put(-35,80){\makebox(0,0)[cc]{\tiny 2}}
\put(5,80){\makebox(0,0)[cc]{\tiny 1}}
\put(45,80){\makebox(0,0)[cc]{\tiny 2}}
\put(85,80){\makebox(0,0)[cc]{\tiny1}}
\put(125,80){\makebox(0,0)[cc]{\tiny 2}}
\put(-100,25){\makebox(0,0)[cc]{\tiny 3}}
\put(-60,25){\makebox(0,0)[cc]{\tiny 3}}
\put(-20,25){\makebox(0,0)[cc]{\tiny 3}}
\put(20,25){\makebox(0,0)[cc]{\tiny 3}}
\put(60,25){\makebox(0,0)[cc]{\tiny 3}}
\put(100,25){\makebox(0,0)[cc]{\tiny 3}}
\end{picture}
\caption{The set $\mathcal D_{0}$.} \label{fig5}
\end{figure}

The case $\lambda \in \Lambda_{3}$ and $a=s_{3}$ gives six pairings
\begin{equation} \label{glue2}
D_{3}^{3} \to D_{0}^{3}, \quad D_{31}^{3} \to D_{1}^{3}, \quad D_{32}^{3} \to D_{2}^{3}, \quad D_{321}^{3} \to D_{21}^{3}, \quad D_{312}^{3} \to D_{12}^{3}, \quad D_{3121}^{3} \to D_{121}^{3}.
\end{equation}
The case $\lambda \in \Lambda_{3} s_{3}$ and $a=s_{1}$ gives three pairings
\begin{equation} \label{glue3}
D_{31}^{1} \to D_{3}^{1}, \quad D_{312}^{1} \to D_{32}^{1}, \quad D_{3121}^{1} \to D_{321}^{1}.
\end{equation}
The gluing given by (\ref{glue1}), (\ref{glue2}) and (\ref{glue3}) leads to the surface $\mathcal A$ which is schematically presented in Figure~\ref{fig6}. Here ideal vertices are indicated by black circles, and other vertices are finite, where four $\pi/2$-angles meet.
\begin{figure}[!ht]
\centering
\unitlength=0.5mm
\begin{picture}(120,120)(-60,0)
\thicklines
\qbezier(0,0)(0,0)(0,120)
\qbezier(0,0)(0,0)(-60,30)
\qbezier(-60,30)(-60,30)(-60,90)
\qbezier(-60,90)(-60,90)(0,120)
\qbezier(0,0)(0,0)(60,30)
\qbezier(60,30)(60,30)(60,90)
\qbezier(60,90)(60,90)(0,120)
\qbezier(0,0)(0,0)(-60,90)
\qbezier(0,0)(0,0)(60,90)
\qbezier(-60,90)(-60,90)(60,90)
\qbezier(0,60)(0,60)(-60,30)
\qbezier(0,60)(0,60)(60,30)
\qbezier(0,60)(0,60)(-60,90)
\qbezier(0,60)(0,60)(60,90)
\put(10,40){\makebox(0,0)[cc]{$D_{0}$}}
\put(-10,40){\makebox(0,0)[cc]{$D_{1}$}}
\put(-20,60){\makebox(0,0)[cc]{$D_{21}$}}
\put(-14,80){\makebox(0,0)[cc]{$D_{121}$}}
\put(14,80){\makebox(0,0)[cc]{$D_{12}$}}
\put(20,60){\makebox(0,0)[cc]{$D_{2}$}}
\put(30,30){\makebox(0,0)[cc]{$D_{3}$}}
\put(-30,30){\makebox(0,0)[cc]{$D_{31}$}}
\put(-50,55){\makebox(0,0)[cc]{$D_{312}$}}
\put(-15,100){\makebox(0,0)[cc]{$D_{3121}$}}
\put(15,100){\makebox(0,0)[cc]{$D_{312}$}}
\put(50,55){\makebox(0,0)[cc]{$D_{32}$}}
\put(-30,110){\makebox(0,0)[cc]{\tiny 2}}
\put(30,110){\makebox(0,0)[cc]{\tiny 2}}
\put(-65,60){\makebox(0,0)[cc]{\tiny 2}}
\put(65,60){\makebox(0,0)[cc]{\tiny 2}}
\put(-30,10){\makebox(0,0)[cc]{\tiny 2}}
\put(30,10){\makebox(0,0)[cc]{\tiny2}}
\put(0,60){\circle*{4}}
\put(0,0){\circle*{4}}
\put(0,120){\circle*{4}}
\put(-60,30){\circle*{4}}
\put(60,30){\circle*{4}}
\put(-60,90){\circle*{4}}
\put(60,90){\circle*{4}}
\end{picture}
\caption{Surface $\mathcal A$.} \label{fig6}
\end{figure}
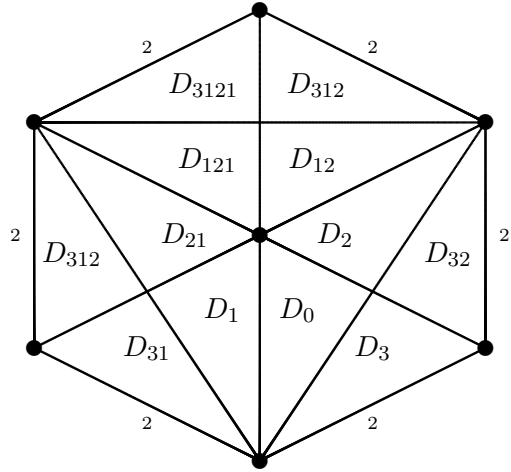

The boundary of the surface $\mathcal A$, presented in Figure~\ref{fig6}, is formed by six sides
\begin{equation} \label{glue4}
D_{3}^{2}, \quad D_{31}^{2}, \quad D_{321}^{2} \quad D_{3121}^{2}, \quad D_{312}^{2}, \quad D_{32}^{2}.
\end{equation}

Now consider further possibilities for $\lambda$ and $a$. The case $\lambda \in \Lambda_{3} s_{3} s_{2}$ and $a = s_{1}$ leads to pairwise identifications as follows:
\begin{equation} \label{glue21}
\begin{gathered}
D_{123}^{1} \to D_{23}^{1}, \qquad D_{1231}^{1} \to D_{231}^{1}, \qquad D_{1232}^{1} \to D_{232}^{1}, \cr  D_{12321}^{1} \to D_{2321}^{1}, \qquad D_{12312}^{1} \to D_{2312}^{1}, \qquad D_{123121}^{1} \to D_{23121}^{1}.
\end{gathered}
\end{equation}
The case $\lambda \in \Lambda_{3} s_{3} s_{2}$ and $a=s_{3}$ leads to the gluing
\begin{equation} \label{glue22}
D_{232}^{3} \to D_{23}^{3}, \qquad D_{2321}^{3} \to D_{231}^{3}, \qquad D_{23121}^{3} \to D_{2312}^{3};
\end{equation}
and the cases $\lambda \in \Lambda_{3} s_{3} s_{2} s_{1}$ with $a \in \{ s_{2}, s_{3} \}$ lead to gluing
\begin{equation} \label{glue23}
\begin{gathered}
D_{1231}^{2} \to D_{123}^{2}, \qquad D_{12312}^{2} \to D_{1232}^{2}, \qquad D_{123212}^{2} \to D_{12321}^{2}, \cr
D_{1232}^{3} \to D_{123}^{3}, \qquad  D_{12321}^{3} \to D_{1231}^{3}, \qquad D_{123121}^{3} \to D_{12312}^{3}.
\end{gathered}
\end{equation}
The gluing given by (\ref{glue21}), (\ref{glue22}) and (\ref{glue23}) lead to the surface $\mathcal B$ presented in Figure~\ref{fig7}.
\begin{figure}[!ht]
\centering
\unitlength=0.5mm
\begin{picture}(120,120)(-60,0)
\thicklines
\qbezier(0,0)(0,0)(0,120)
\qbezier(0,0)(0,0)(-60,30)
\qbezier(-60,30)(-60,30)(-60,90)
\qbezier(-60,90)(-60,90)(0,120)
\qbezier(0,0)(0,0)(60,30)
\qbezier(60,30)(60,30)(60,90)
\qbezier(60,90)(60,90)(0,120)
\qbezier(0,0)(0,0)(-60,90)
\qbezier(0,0)(0,0)(60,90)
\qbezier(-60,90)(-60,90)(60,90)
\qbezier(0,60)(0,60)(-60,30)
\qbezier(0,60)(0,60)(60,30)
\qbezier(0,60)(0,60)(-60,90)
\qbezier(0,60)(0,60)(60,90)
\put(-12,40){\makebox(0,0)[cc]{$D_{12312}$}}
\put(-20,60){\makebox(0,0)[cc]{$D_{123121}$}}
\put(-14,80){\makebox(0,0)[cc]{$D_{12321}$}}
\put(14,80){\makebox(0,0)[cc]{$D_{1231}$}}
\put(20,60){\makebox(0,0)[cc]{$D_{123}$}}
\put(12,40){\makebox(0,0)[cc]{$D_{1232}$}}
\put(-30,30){\makebox(0,0)[cc]{$D_{2312}$}}
\put(-50,55){\makebox(0,0)[cc]{$D_{23121}$}}
\put(-15,100){\makebox(0,0)[cc]{$D_{2321}$}}
\put(15,100){\makebox(0,0)[cc]{$D_{231}$}}
\put(50,55){\makebox(0,0)[cc]{$D_{23}$}}
\put(30,30){\makebox(0,0)[cc]{$D_{232}$}}
\put(-30,110){\makebox(0,0)[cc]{\tiny 2}}
\put(30,110){\makebox(0,0)[cc]{\tiny 2}}
\put(-65,60){\makebox(0,0)[cc]{\tiny 2}}
\put(65,60){\makebox(0,0)[cc]{\tiny 2}}
\put(-30,10){\makebox(0,0)[cc]{\tiny 2}}
\put(30,10){\makebox(0,0)[cc]{\tiny2}}
\put(0,60){\circle*{4}}
\put(0,0){\circle*{4}}
\put(0,120){\circle*{4}}
\put(-60,30){\circle*{4}}
\put(60,30){\circle*{4}}
\put(-60,90){\circle*{4}}
\put(60,90){\circle*{4}}
\end{picture}
\caption{Surface $\mathcal B$.} \label{fig7}
\end{figure}
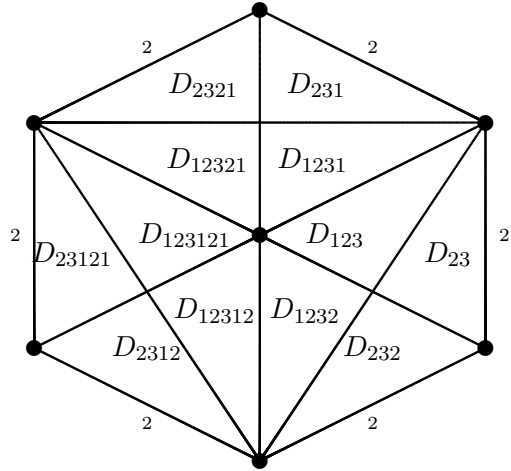

The boundary of the surface $\mathcal B$, presented in Figure~\ref{fig7}, is formed by six sides
\begin{equation} \label{glue24}
D_{2312}^{2}, \quad D_{23123}^{2}, \quad D_{2321}^{2} \quad D_{231}^{2}, \quad D_{23}^{2}, \quad D_{232}^{2}.
\end{equation}
Sides listed in (\ref{glue4}) and (\ref{glue24}) are pairwise glued by generators of $PT_{4}$ which appear in the case $\lambda \in \Lambda_{3} s_{3}$ and $a = s_{2}$
\begin{equation}\label{glue25}
\begin{gathered}
D_{23}^{2} \to D_{3}^{2}, \qquad D_{231}^{2} \to D_{31}^{2}, \qquad D_{232}^{2} \to D_{32}^{2}, \cr D_{2321}^{2} \to D_{321}^{2}, \qquad D_{2312}^{2} \to D_{312}^{2}, \qquad D_{23121}^{2} \to D_{3121}^{2}.
\end{gathered}
\end{equation}
Thus, after gluing boundary sides by (\ref{glue25}), we get the surface $\mathcal C = \mathcal A \cup \mathcal B$ that is homeomorphic to the 2-sphere with eight punctures corresponding to ideal vertices. By the preceding construction $PT_{4} = \pi_{1} (\mathcal C)$, and hence $PT_4$ is a free group of rank seven.
\vspace*{1mm}

\section{Generators of pure twin groups} \label{sec4}
In this section, we determine generators of $PT_n$ and give an upper bound on its rank.
\par

\subsection{Generators and rank of $T_n$}
We begin with the following useful result which is a consequence of direct computations using relations in $T_n$.

\begin{lemma}\label{key-lemma}
The following holds:
\begin{enumerate}
\item If $~2 \le l \le n-1$, then $$s_{l-1}\big((s_{l-1} s_l)^3 \big)^{s_{l+1}s_{l+2}\ldots s_{n-1}}=\big((s_ls_{l-1})^3 \big)^{s_{l+1}s_{l+2}\ldots s_{n-1}} s_{l-1}.$$
\item If $~ i_{n-1}+1< l \le n-1$, then $$m_{n-1,i_{n-1}} s_l ~m_{n-1,i_{n-1}}^{-1} s_{l-1}= \big((s_l s_{l-1})^3 \big)^{s_{l+1}s_{l+2}\ldots s_{n-1}}.$$
\item If $~l=n-1$, then $$m_{n-1,i_{n-1}} s_l ~m_{n-1,i_{n-1}}^{-1} s_{l-1}= (s_l s_{l-1})^3.$$
\item  If $~ i_{n-1}+1< l \le n-1$, then $$\overline{m_{n-1,i_{n-1}}~ s_l} = \overline{s_{l-1}~m_{n-1,i_{n-1}} }.$$
\end{enumerate}
\end{lemma}

\begin{proof}
For (1), consider
\begin{eqnarray*}
s_{l-1}\big((s_{l-1} s_l)^3 \big)^{s_{l+1}s_{l+2}\ldots s_{n-1}} &=& s_{l-1} \big(s_{n-1} \ldots s_{l+2} s_{l+1} (s_{l-1} s_l)^3 s_{l+1}s_{l+2}\ldots s_{n-1} \big)\\
&=& s_{n-1} \ldots s_{l+2} s_{l+1} s_{l-1} (s_{l-1} s_l)^3 s_{l+1}s_{l+2}\ldots s_{n-1}\\
&=& s_{n-1} \ldots s_{l+2} s_{l+1}  (s_l s_{l-1})^3 s_{l-1} s_{l+1}s_{l+2}\ldots s_{n-1}\\
&=& \big(s_{n-1} \ldots s_{l+2} s_{l+1}  (s_l s_{l-1})^3  s_{l+1}s_{l+2}\ldots s_{n-1} \big) s_{l-1}\\
&=& \big((s_l s_{l-1})^3 \big)^{s_{l+1}s_{l+2}\ldots s_{n-1}} s_{l-1}.
\end{eqnarray*}

Next, for $i_{n-1}+1< l \le n-1$, we have
\begin{small}
\begin{eqnarray*}
m_{n-1,i_{n-1}} s_l ~m_{n-1,i_{n-1}}^{-1} s_{l-1} &=& s_{n-1} \ldots s_{l+1} s_l s_{l-1} \ldots s_{i_{n-1}+1} s_l s_{i_{n-1}+1} \ldots  s_{l-1}  s_l s_{l+1} \ldots s_{n-1} s_{l-1}\\
&=& s_{n-1} \ldots s_{l+1} s_l s_{l-1} s_l \ldots s_{i_{n-1}+1} s_{i_{n-1}+1} \ldots  s_{l-1}  s_l s_{l+1} \ldots s_{n-1} s_{l-1}\\
&=& s_{n-1} \ldots s_{l+1} s_l s_{l-1} s_l  s_{l-1}  s_l s_{l+1} \ldots s_{n-1} s_{l-1}\\
&=& s_{n-1} \ldots s_{l+1} s_l s_{l-1} s_l  s_{l-1}  s_l  s_{l-1} s_{l+1} \ldots s_{n-1}\\
&=& \big( (s_l s_{l-1})^3 \big)^{s_{l+1} \ldots s_{n-1}}.
\end{eqnarray*}
\end{small}
This proves (2), and (3) is its special case.
\par

Let $i_{n-1}+1< l \le n-1$. Then, for (4), we have
\begin{eqnarray*}
\overline{m_{n-1,i_{n-1}}~ s_l} &=& \overline{s_{n-1} \ldots s_{l+1} s_l s_{l-1} s_{l-2}\ldots s_{i_{n-1}+1} s_l}\\
&=& \overline{s_{n-1} \ldots s_{l+1} s_l s_{l-1}  s_l s_{l-2} \ldots s_{i_{n-1}+1}}\\
&=& \overline{s_{n-1} \ldots s_{l+1} s_{l-1}  s_l s_{l-1} s_{l-2} \ldots s_{i_{n-1}+1}},~\textrm{using the braid relations}\\
&=& \overline{s_{l-1}~ s_{n-1} \ldots s_{l+1}   s_l s_{l-1} s_{l-2} \ldots s_{i_{n-1}+1}}\\
&=& \overline{s_{l-1}~ m_{n-1,i_{n-1}}}.
\end{eqnarray*}
\end{proof}

Observe that $PT_{n-1}$ is a subgroup of $PT_n$ for each $n \ge 2$.  In the next result, we determine generators of $PT_n$ that do not lie in $PT_{n-1}$.
\par

\begin{theorem}\label{pure-twin-gen}
The pure twin group $PT_n$, for $n>2$, is generated by $PT_{n-1}$ together with the following additional generators
 $$\big\{\big((s_{l-1} s_l)^3 \big)^{(\lambda_0 ~m_{n-1, l})^{-1}}~|~  2 \le l \le n-1~\textrm{and}~\lambda_0 \in \Lambda_{n-1} \big\}.$$
\end{theorem}

\begin{proof}
By Reidemeister--Schreier Theorem, the pure twin group $PT_n$ is generated by the set
$$\big\{S_{\lambda, a}= (\lambda a)(\overline{\lambda a})^{-1}~|~\lambda \in \Lambda_n~\textrm{and}~a \in \{s_1, s_1, \ldots, s_{n-1} \} \big\}.$$ Let $\lambda \in \Lambda_n$, and $a=s_l$ for some $1\le l \le n-1$. Then we can write $\lambda = \lambda_0 m_{n-1, i_{n-1}}$, where $$ \lambda_0= m_{1, i_{1}}m_{2, i_{2}}\cdots m_{n-2, i_{n-2}}.$$
\par
We now have the following four cases:
\par

Case (1): If $\lambda_0=1$ and $m_{n-1, i_{n-1}}=1$, then  $S_{\lambda, a}=1$.
\par

Case (2): If $\lambda_0\neq 1$ and $m_{n-1, i_{n-1}}=1$, then  there are two subcases:
\par

Subcase (2a): If $l =n-1$, then
\begin{eqnarray*}
S_{\lambda_0, s_l} &=& (\lambda_0 s_{n-1})(\overline{\lambda_0 s_{n-1}})^{-1}\\
&=&   (\lambda_0 s_{n-1})(\lambda_0 s_{n-1})^{-1}\\
&=& 1.
\end{eqnarray*}
\par
Subcase (2b): If $l < n-1$, then $S_{\lambda_0, s_l} \in PT_{n-1}$.
\par

Case (3): If $\lambda_0=1$ and $m_{n-1, i_{n-1}} \neq 1$, then we have the following four cases:
\par

Subcase  (3a): If $l>  i_{n-1}+1$, then a direct computation yields
\begin{eqnarray*}
S_{m_{n-1, i_{n-1}}, s_l} &=& (m_{n-1, i_{n-1}} s_l)(\overline{m_{n-1, i_{n-1}} s_l})^{-1}\\
&=&  (m_{n-1, i_{n-1}} s_l) (s_{l-1} m_{n-1, i_{n-1}})^{-1}\\
&=& \big((s_l s_{l-1})^3\big)^{s_{l+1}s_{l+2}\cdots s_{n-1}},~\textrm{using~ Lemma~\ref{key-lemma}(2)}\\
&=& \big((s_l s_{l-1})^3\big)^{m_{n-1, l}^{-1}}.
\end{eqnarray*}
\par

Subcase (3b): If $l =  i_{n-1}+1$, then
\begin{eqnarray*}
S_{m_{n-1, i_{n-1}}, s_l} &=& (m_{n-1, i_{n-1}} s_{i_{n-1}+1})(\overline{m_{n-1, i_{n-1}} s_{i_{n-1}+1}})^{-1}\\
&=&  m_{n-1, i_{n-1}+1} ~\overline{m_{n-1, i_{n-1}+1}}^{-1}\\
&=& 1.
\end{eqnarray*}
\par

Subcase (3c): If $l =  i_{n-1}$, then
\begin{eqnarray*}
S_{m_{n-1, i_{n-1}}, s_l} &=& (m_{n-1, i_{n-1}} s_{i_{n-1}})(\overline{m_{n-1, i_{n-1}} s_{i_{n-1}}})^{-1}\\
&=&  m_{n-1, i_{n-1}-1} ~\overline{m_{n-1, i_{n-1}-1}}^{-1}\\
&=& 1.
\end{eqnarray*}
\par

Subcase  (3d): If $l <  i_{n-1}$, then
\begin{eqnarray*}
S_{m_{n-1, i_{n-1}}, s_l} &=& (m_{n-1, i_{n-1}} s_l)(\overline{m_{n-1, i_{n-1}} s_l})^{-1}\\
&=&  (s_l m_{n-1, i_{n-1}})(\overline{ s_l m_{n-1, i_{n-1}}})^{-1}\\
&=& (s_l m_{n-1, i_{n-1}})(s_l m_{n-1, i_{n-1}})^{-1}\\\
&=& 1.
\end{eqnarray*}
\par

Case (4): If $\lambda_0 \neq1$ and $m_{n-1, i_{n-1}} \neq 1$, then we have the following four cases:
\par

Subcase (4a): If $l>  i_{n-1}+1$, then
\begin{eqnarray*}
S_{\lambda_0 m_{n-1, i_{n-1}}, s_l} &=& (\lambda_0 m_{n-1, i_{n-1}} s_l)(\overline{\lambda_0 m_{n-1, i_{n-1}} s_l})^{-1}\\
&=&  (\lambda_0 m_{n-1, i_{n-1}} s_l) (\overline{\lambda_0 s_{l-1} m_{n-1, i_{n-1}}})^{-1}\\
&=&  (\lambda_0 m_{n-1, i_{n-1}} s_l) (m_{n-1, i_{n-1}}^{-1}) (\overline{\lambda_0 s_{l-1}})^{-1}\\
&=& (\lambda_0 s_{l-1}) \big((s_{l-1} s_l)^3 \big)^{s_{l+1}s_{l+2}\cdots s_{n-1}}(\overline{\lambda_0 s_{l-1}})^{-1},~\textrm{using~ Lemma~\ref{key-lemma}(2)}\\
&=& \big((s_{l-1} s_l)^3 \big)^{s_{l+1}s_{l+2}\cdots s_{n-1} (\lambda_0 s_{l-1})^{-1}} (\lambda_0 s_{l-1}) (\overline{\lambda_0 s_{l-1}})^{-1}\\
&=& \big((s_l s_{l-1})^3 \big)^{s_{l+1}s_{l+2}\cdots s_{n-1} \lambda_0^{-1}} ~S_{\lambda_0, s_{l-1}}\\
&=& \big((s_l s_{l-1})^3 \big)^{(\lambda_0 ~m_{n-1, l})^{-1}} ~S_{\lambda_0, s_{l-1}},
\end{eqnarray*}
where $S_{\lambda_0, s_{l-1}} \in PT_{n-1}$.
\par

Subcase  (4b): If $l =  i_{n-1}+1$, then
\begin{eqnarray*}
S_{\lambda_0 m_{n-1, i_{n-1}}, s_l} &=& (\lambda_0 m_{n-1, i_{n-1}} s_{i_{n-1}+1})(\overline{\lambda_0 m_{n-1, i_{n-1}} s_{i_{n-1}+1}})^{-1}\\
&=& (\lambda_0 m_{n-1, i_{n-1}+1}) (\overline{ \lambda_0 m_{n-1, i_{n-1}+1}})^{-1}\\
&=& 1.
\end{eqnarray*}
\par

Subcase  (4c): If $l =  i_{n-1}$, then
\begin{eqnarray*}
S_{\lambda_0 m_{n-1, i_{n-1}}, s_l} &=& (\lambda_0 m_{n-1, i_{n-1}} s_{i_{n-1}})(\overline{\lambda_0 m_{n-1, i_{n-1}} s_{i_{n-1}}})^{-1}\\
&=&  (\lambda_0 m_{n-1, i_{n-1}-1}) (\overline{\lambda_0 m_{n-1, i_{n-1}-1}})^{-1}\\
&=& 1.
\end{eqnarray*}
\par

Subcase  (4d): If $l <  i_{n-1}$, then
\begin{eqnarray*}
S_{\lambda_0 m_{n-1, i_{n-1}}, s_l} &=& (\lambda_0 m_{n-1, i_{n-1}} s_l)(\overline{\lambda_0 m_{n-1, i_{n-1}} s_l})^{-1}\\
&=&  (\lambda_0 s_l m_{n-1, i_{n-1}})(\overline{\lambda_0  s_l m_{n-1, i_{n-1}}})^{-1}\\
&=&  (\lambda_0 s_l m_{n-1, i_{n-1}})(m_{n-1, i_{n-1}})^{-1}(\overline{\lambda_0  s_l})^{-1}\\
&=& (\lambda_0 s_l)(\overline{\lambda_0  s_l})^{-1}\\
&=& S_{\lambda_0, s_l} \in PT_{n-1}.
\end{eqnarray*}
Note that Subcase (4a) yields the case Subcase (3a) by taking $\lambda_0=1$. This completes the proof of the theorem.
\end{proof}
\par

\begin{remark}
By Theorem \ref{pure-twin-gen}, for $n \ge 6$, $(s_1s_2)^3$ and $(s_4s_5)^3$ are two distinct commuting generators of $PT_n$. This shows that $PT_n$ is not free for $n \ge 6$, giving an alternate proof of the first part of Theorem \ref{pure-twin-not-free}.
\end{remark}
\par

The following result is useful for determining the minimal number of generators of $PT_n$.
\par

\begin{lemma}\label{inverse-coset-representative}
If $\lambda \in \Lambda_n$, then $\lambda^{-1}\in \Lambda_n$.
\end{lemma}
\begin{proof}
If $\lambda= m_{k,i_k}$ for some $1 \le k \le n-1$, then $$m_{k,i_k}^{-1}=s_{i_k+1}\ldots s_k=m_{i_k+1, i_k}\ldots m_{k, k-1} \in \Lambda_n,$$  and the result holds. We note that it is enough to prove the assertion for Schreier cosets of the form $\lambda= m_{k,i_k}m_{l,i_l}$ for some $1 \le k < l \le n-1$. We prove this by induction on $n$. The result obviously holds for $n=1,2$, which is the base step of the induction. Now consider
\begin{eqnarray*}
\lambda^{-1} &=& m_{l,i_l}^{-1}m_{k,i_k}^{-1}\\
&=& (s_l s_{l-1}\ldots s_{i_l+1})^{-1} (s_k s_{k-1}\ldots s_{i_k+1})^{-1}\\
&=& s_{i_l+1} \ldots  s_{l-1} s_l ~ s_{i_k+1} \ldots s_{k-1} s_k.
\end{eqnarray*}
\par
Now, if $k +1= l$, then
\begin{eqnarray*}
\lambda^{-1} &=& s_{i_l+1} \ldots  s_{l-1} ~ s_{i_k+1} \ldots s_{k-1} (s_{k+1} s_k)\\
&=& (m_{k-1, i_k}m_{l-1, i_l})^{-1} (s_{k+1} s_k) \in \Lambda_n,
\end{eqnarray*}
since $(m_{k-1, i_k}m_{l-1, i_l})^{-1}  \in \Lambda_n$ by induction hypothesis.
\par
If $k +2 \le l$, then
\begin{eqnarray*}
\lambda^{-1} &=& (s_{i_l+1} \ldots  s_{l-1} ~ s_{i_k+1} \ldots s_{k-1} s_k) s_l\\
&=& (m_{k, i_k}m_{l-1, i_l})^{-1} s_l \in \Lambda_n,
\end{eqnarray*}
since $(m_{k, i_k}m_{l-1, i_l})^{-1}  \in \Lambda_n$ by induction hypothesis. This completes the proof.
\end{proof}
\par

\begin{remark}\label{pure-twin-generator-remark}
In view of Lemma \ref{inverse-coset-representative}, if $\lambda_0 \in \Lambda_{n-1}$, then $\lambda_0^{-1} \in \Lambda_{n-1}$. Suppose that $\lambda_0^{-1} = m_{1, i_1}m_{2, i_2}\cdots m_{n-3, i_{n-3}}m_{n-2, i_{n-2}}$. Then, for $2 \le l \le n-1$, we have
\begin{eqnarray*}
(\lambda_0 ~m_{n-1, l})^{-1} &=& m_{n-1, l}^{-1} ~\lambda_0^{-1}\\
&=& s_{l+1}s_{l+2}\cdots s_{n-2}s_{n-1} ~m_{1, i_1}m_{2, i_2}\cdots m_{n-3, i_{n-3}}m_{n-2, i_{n-2}}\\
&=& s_{l+1}s_{l+2}\cdots s_{n-2} ~m_{1, i_1}m_{2, i_2}\cdots m_{n-3, i_{n-3}} (s_{n-1} m_{n-2, i_{n-2}})\\
& & \vdots\\
&=& m_{1, i_1}m_{2, i_2} \cdots m_{l-2, i_{l-2}} m_{l-1, i_{l-1}} (s_{l+1} m_{l, i_l}) \cdots  (s_{n-2} m_{n-3, i_{n-3}}) (s_{n-1} m_{n-2, i_{n-2}}).
\end{eqnarray*}
Hence, $(\lambda_0 ~m_{n-1, l})^{-1}$ is a Schreier coset representative not containing the term $m_{l, i_l}$. Thus, by Theorem \ref{pure-twin-gen}, for $n>2$, generators of $PT_n$ that do not lie in $PT_{n-1}$ are given by
\begin{equation}\label{refined-pure-twin-gen}
\big\{\big((s_{l-1} s_l)^3 \big)^{\lambda}~|~ 2 \le l \le n-1~\textrm{and}~\lambda \in \Lambda_n~\textrm{does not contain the term}~m_{l, i_l}\big\}.
\end{equation}
\end{remark}
\vspace*{1mm}

A direct computation shows that \eqref{equation-generalize} holds in general.
\par

\begin{lemma}\label{equation-generalize-higher}
The following relation holds in $PT_n$ for $n \ge 4$:
$$ (s_{n-2}s_{n-3})^3 \big((s_{n-2}s_{n-1})^3 \big)^{m_{n-3, n-4}\,m_{n-2, n-4}}(s_{n-3}s_{n-2})^3= \big((s_{n-1}s_{n-2})^3\big)^{m_{n-3, n-4} \,m_{n-2, n-3}}.$$
\end{lemma}
\vspace*{1mm}

\begin{theorem}\label{rank-pure-twin}
For $n \ge 5$, rank of $PT_n$ is at most $r_n$, where $$r_n=r_{n-1} + \frac{(n-1)!}{2}+ (n-1)! \Big(\sum_{l=3}^{n-2} \frac{(l-1)^2}{l !}\Big)+ (n-3)(n-1)$$
and $r_4=7$.
\end{theorem}

\begin{proof}
We proceed as per the following cases:
\par
Case(1): For $l=2$, generators of $PT_n$ that do not lie in $PT_{n-1}$ are given by
$$
\big\{\big((s_1 s_2)^3 \big)^{\lambda}~|~\lambda =m_{3, i_3}m_{4, i_4} \cdots m_{n-1, i_{n-1}}\big\}.
$$
Notice that $m_{k, i_k}\neq1$  for each $3 \le k \le n-1$. For, if some $m_{k, i_k}=1$, then $\big((s_1 s_2)^3 \big)^{\lambda} \in PT_{n-1}$. Thus, there are precisely $(n-1)!/2$ such generators.
\par

Case(2): For $2 < l <n-1$, in view of Remark \eqref{pure-twin-generator-remark}, generators of $PT_n$ that do not lie in $PT_{n-1}$ are given by
$$
\big\{\big((s_{l-1} s_l)^3 \big)^{\lambda}~|~\lambda =m_{l-2, i_{l-2}}m_{l-1, i_{l-1}} m_{l+1, i_{l+1}}m_{l+2, i_{l+2}}\cdots m_{n-1, i_{n-1}}\big\}.
$$
Just as in Case (1), $m_{k, i_k}\neq1$  for each $l+1 \le k \le n-1$. If $m_{l-2, i_{l-2}}=1=m_{l-1, i_{l-1}}$, then the number of generators is $(l+1)(l+2)\cdots (n-1)=(n-1)!/l!$. In the other case, $m_{l-2, i_{l-2}} \neq 1$ but $m_{l-1, i_{l-1}}$ could be trivial, and hence the number of generators is $(l-2) l (l+1)(l+2)\cdots (n-1)=(n-1)!/\big((l-3)! (l-1)\big)$.
\par

Case(3): For $l =n-1$, we have $m_{n-1, l}=1$, and hence generators are of the form
$$
\big\{\big((s_{n-2} s_{n-1})^3 \big)^{\lambda}~|~\lambda =m_{n-3, i_{n-3}} m_{n-2, i_{n-2}}\big\}.
$$
The number of such generators is $(n-3)(n-1)+1$. But, in view of Lemma \ref{equation-generalize-higher}, two of these generators are conjugated by a generator lying in $PT_{n-1}$, and hence the actual number is $(n-3)(n-1)$.
Set $r_4=7$, and for $n \ge 5$, define
\begin{eqnarray*}
r_n  &= & r_{n-1} + \frac{(n-1)!}{2}+\sum_{l=3}^{n-2} \Big(\frac{(n-1)!}{l!} + \frac{(n-1)!}{(l-3)! (l-1)} \Big)  + (n-3)(n-1)\\
&= & r_{n-1} + \frac{(n-1)!}{2}+ (n-1)! \Big(\sum_{l=3}^{n-2} \frac{(l-1)^2}{l !}\Big)+ (n-3)(n-1).
\end{eqnarray*}
Then our computations show that rank of $PT_n$ is at most $r_n$.
\end{proof}
\vspace*{1mm}

\begin{remark}
By \cite{Barcelo} and \cite{BW}, the first Betti number of the Eilenberg--Maclane space $X_n$ is $2^{n-3}(n^2-5n+8)-1$, which is a lower bound for the number of generators of $PT_n$ for $n >2$.
\end{remark}


\subsection{Conjugation action of $T_n$ on $PT_n$}
Since $PT_n$ is a normal subgroup of $T_n$, there is a natural homomorphism $$\phi_n: \Inn(T_n) \to \Aut(PT_n)$$ obtained by restricting inner automorphisms of $T_n$ to $PT_n$. In view of Corollary \ref{center-twin}, $T_n \cong \Inn(T_n)$ for $n >2$, and we have $$\phi_n: T_n \to \Aut(PT_n).$$ By Theorem \ref{pure-twin-n23},  $PT_3 \cong \mathbb{Z}$ and we have
$$
\left( (s_1 s_2)^3 \right)^{s_1} = (s_2 s_1)^3~\textrm{and}~ \left( (s_1 s_2)^3 \right)^{s_2} = (s_2 s_1)^3.
$$
Thus, the homomorphism $\phi_3: T_3 \to \Aut(\mathbb{Z})$ is not faithful.

Again, by Theorem \ref{pure-twin-n23}, $PT_4\cong F_7$ is generated by the elements
$$ b_1 = (s_1 s_2)^3,\quad b_2 = \left( (s_1 s_2)^3 \right)^{s_3}, \quad b_3 = \big( (s_1 s_2)^3\big)^{s_3s_2},\quad b_4 = \left( (s_1 s_2)^3 \right)^{s_3 s_2 s_1},$$
$$ \quad b_5 = (s_2 s_3)^3,  \quad  b_6 = \left((s_2 s_3)^3 \right)^{s_1}, \quad b_7 = \big((s_2 s_3)^3\big)^{s_1s_2}.$$

A direct computation shows that the automorphisms $\phi_4(s_i)$, $i = 1, 2, 3$, act on these generators by the rules:

$$
\phi_4(s_1) : \left\{
\begin{array}{l}
b_1 \mapsto b_1^{-1}, \\
b_2 \mapsto b_2^{-1}, \\
b_3 \mapsto b_4, \\
b_4 \mapsto b_3, \\
b_5 \mapsto b_6, \\
b_6 \mapsto b_5, \\
b_7 \mapsto b_1 b_7^{-1} b_1^{-1},
\end{array}
\right.~~~
\phi_4(s_2) : \left\{
\begin{array}{l}
b_1 \mapsto b_1^{-1}, \\
b_2 \mapsto b_3, \\
b_3 \mapsto b_2, \\
b_4 \mapsto b_1^{-1} b_4^{-1} b_1, \\
b_5 \mapsto b_5^{-1}, \\
b_6 \mapsto b_7, \\
b_7 \mapsto b_6,
\end{array}
\right.
$$
and
$$
\phi_4(s_3) : \left\{
\begin{array}{l}
b_1 \mapsto b_2, \\
b_2 \mapsto b_1, \\
b_3 \mapsto b_5^{-1} b_3^{-1} b_5, \\
b_4 \mapsto b_6^{-1} b_4^{-1} b_6, \\
b_5 \mapsto b_5^{-1}, \\
b_6 \mapsto b_6^{-1}, \\
b_7 \mapsto b_2^{-1}b_6^{-1} b_{4}^{-1} b_{1} b_7 b_3 b_5.
\end{array}
\right.
$$
\par

Let $V$ be the verbal subgroup of $PT_4$ defined by the word $w = x^2$. Then $$PT_4 / V \cong \mathbb{Z}_2^{\oplus 7},$$ the direct product of 7 copies of cyclic group of order 2. For each $i=1,2, \ldots, 7$, let $\beta_i$ denote the image of $b_i$ in $\mathbb{Z}_2^{\oplus 7}$. The automorphisms $\phi_4(s_i)$ induce automorphisms of $\mathbb{Z}_2^{\oplus 7}$. If we restrict these automorphisms onto the subgroup $G_4 = \langle \beta_1, \beta_2, \beta_3, \beta_4 \rangle$, then they act as permutations on the generators. In particular, $\phi_4(s_1)$ acts as permutation $(3 4)$, i.e. permutes generators $\beta_3$ and $\beta_4$; $\phi_4(s_2)$ acts as permutation $(2 3)$ and  $\phi_4(s_3)$ acts as permutation $(1 2)$. Hence, the group $\big\langle \phi_4 (s_1), \phi_4(s_2), \phi_4(s_3) \big\rangle$ acts on $G_4$ as $S_4$. Now we can prove
\par

\begin{prop}\label{pt4-faithful}
The representation  $\phi_4: T_4 \to \Aut(F_7)$ is faithful.
\end{prop}

\begin{proof}
Consider the extension $1 \to PT_4 \to T_4 \to S_4 \to 1$. Each element of $T_4$ can be written uniquely as $a \lambda$ for some $a \in PT_4$ and $\lambda \in \Lambda_4$. Now, suppose that $\phi_4(a \lambda)=\id$, the identity automorphism. Then it induces an automorphism of $G_4$. It is trivial only in the case when $\lambda$ is trivial. Hence, $a \lambda = a$ and $\phi_4(a)$ is the conjugation by $a$ which implies that $a=1$. Hence, the representation $\phi_4$ is faithful.
\end{proof}
\vspace*{1mm}

\begin{problem}
The following problems remain unsettled at this point:
\begin{enumerate}
\item Is it true that   $\text{\rm Ker}(\phi_n)$  is trivial for $n \ge 5$?
\item A group $G$ is residually nilpotent if $ \bigcap_{i=1}^{\infty} \gamma_i (G) = 1$, where $\gamma_i (G)$ is the $i$-th term of the lower central series of $G$. Obviously,  $T_2 = \mathbb{Z}_2$ is residually nilpotent. Since $T_3 = \mathbb{Z}_2 *\mathbb{Z}_2$ is the infinite dihedral group, we have $\gamma_{2} (T_3) = \big\langle (s_1 s_2)^2 \big\rangle$, and hence
$$
\gamma_{i} (T_3) = \big\langle (s_1 s_2)^{2^{i-1}} \big\rangle~\textrm{for}~i \geq 2.
$$
Thus,  $\bigcap_{i=1}^{\infty} \gamma_i (T_3) = 1$ and $T_3$ is residually nilpotent. Determine whether $T_n$ is residually nilpotent for $n \ge 4$.
\end{enumerate}
\end{problem}
\vspace*{1mm}

\section{Virtual and welded twin groups} \label{sec5}

We introduce virtual twin groups $VT_n$ and welded twin groups $WT_n$ in analogy with virtual and welded braid groups. By  \emph{virtual twin group} $VT_n$ we mean the group generated by elements
$$
s_1, s_2, \ldots, s_{n-1}, \rho_1, \rho_2, \ldots, \rho_{n-1},
$$
with defining relations (\ref{r1})--(\ref{r2}) and (\ref{r3})--(\ref{r7}), where
\begin{align}
  \rho_{i}^{2} &= 1 ~~ \hspace{3,5cm} \mbox{for}~~ i = 1, 2, \ldots,n-1,\label{r3}\\
 \rho_{i}\rho_{j} &= \rho_{j}\rho_{i} ~~      \hspace{3cm} \mbox{for}~~ |i-j|\geq 2,\label{r4}\\
 \rho_{i}\rho_{i+1}\rho_{i} &= \rho_{i+1}\rho_{i}\rho_{i+1}~~ \hspace{2cm} \mbox{for}~~ i = 1, 2 \ldots,n-2,\label{r5}\\
 s_{i} \rho_{j}& = \rho_{j} s_{i} ~~      \hspace{3cm} \mbox{for}~~ |i-j|\geq 2,\label{r6}\\
 \rho_{i}\rho_{i+1}s_{i} &= s_{i+1}\rho_{i}\rho_{i+1}~~ \hspace{2cm} \mbox{for}~~ i = 1, 2, \ldots,n-2.\label{r7}
 \end{align}

The quotient of $VT_n$ by additional relations
\begin{align}
\rho_{i} s_{i+1}  s_{i} &=s_{i+1} s_i \rho_{i+1}  \hspace{0,3cm} \mbox{for}~~ i = 1, 2,\ldots,n-2,\label{r9}
 \end{align}
is called the \emph{welded twin group}  and denoted by $WT_n$. Notice that the second forbidden relation $s_i s_{i+1} \rho_i = \rho_{i+1} s_i s_{i+1}$ also holds in $WT_n$.
\par

The kernel of the homomorphism
$$
VT_n \longrightarrow S_n,
$$
which maps $s_i$ to $\rho_i$ and $\rho_i$ to $\rho_i$ for $i=1,2, \ldots, n-1$ is the \emph{pure virtual twin group} $PVT_n$. This homomorphism induces the homomorphism
$$
WT_n \longrightarrow S_n
$$
whose kernel is the \emph{pure welded twin group} $PWT_n$.
\par

\begin{prop} The following holds:
\begin{enumerate}
\item  $VT_n = PVT_n \leftthreetimes S_n$.
\item $WT_n = PWT_n \leftthreetimes S_n$.
\item $T_n$ is a subgroup of $VT_n$.
\end{enumerate}
\end{prop}
\par

\begin{proof}
The preceding homomorphisms are split-surjections establishing the assertions (1) and (2). Consider the  endomorphism of $VT_n$ defined on the generators by
$$
\rho_i \mapsto 1~\textrm{and}~ s_i \mapsto s_i
$$
for $ i = 1, 2, \ldots, n-1$. The image of this endomorphism is $T_n$, which proves (3).
\end{proof}
\par

\begin{problem}
We conclude with the following problems:
\begin{enumerate}
\item Find presentations for the groups $PVT_n$ and $PWT_n$.
\item Find presentations for the commutator subgroups of $VT_n$ and $WT_n$. The commutator subgroup of $T_n$ has been investigated in \cite{DG}.
\item Determine whether $PVT_n$ and $PWT_n$ are residually nilpotent (residually finite).
\item Determine whether $VT_n$ and $WT_n$ are linear.
\end{enumerate}
\end{problem}
\par

\begin{ack}
After this paper was submitted to the journal and uploaded on the arxiv, we were informed by Harshman and Knapp about their preprint \cite{HK}. They refer twin and pure twin groups as triad and pure triad groups, respectively, and explore interesting relations of these groups with three-body hard-core interactions in one dimension. Bardakov and Vesnin are supported by the Russian Science Foundation grant 16-41-02006. Singh is supported by the DST-RSF grant INT/RUS/RSF/P-2 and SERB MATRICS Grant MTR/2017/000018.
\end{ack}
\par

\end{document}